\documentclass[11pt]{amsart}
\usepackage{amsmath,amsfonts,amssymb,amsthm,aliascnt}
\usepackage{latexsym,bm,graphicx}
\usepackage{mathrsfs}
\usepackage{color}
\usepackage[linktocpage, colorlinks,  citecolor = green, linkcolor = blue]{hyperref}
\usepackage[all]{xy}

\usepackage{enumitem}

\newcommand{\customitemize}[1]{%
	\begin{enumerate}[label={(#1\arabic*)}, ref=#1\arabic*]
	}

\title[Splitting]{Splitting and Slow Volume Growth for Open Manifolds with Nonnegative Ricci Curvature}
\author{Hongzhi Huang and Xian-Tao Huang}
\address{Hongzhi Huang \\ Department of Mathematics \\ Jinan University\\ Guangzhou 510632}
\email{\href{mailto:huanghz@jnu.edu.cn}{huanghz@jnu.edu.cn}
}

\address{Xian-Tao Huang\\School of Mathematics\\  Sun Yat-sen University\\ Guangzhou 510275}
\email{\href{mailto:hxiant@mail2.sysu.edu.cn}{hxiant@mail2.sysu.edu.cn}}

\newtheorem{thm}{Theorem}[section]
\newtheorem{prop}[thm]{Proposition}
\newtheorem{lem}[thm]{Lemma}
\newtheorem{slem}[thm]{Sublemma}

\newtheorem{cor}[thm]{Corollary}

\newtheorem{mainthm}{Theorem}[section] 
\newaliascnt{myMainThm}{mainthm}        
\newtheorem*{myMainThm*}{Theorem \Alph{mainthm}} 

\newenvironment{maintheorem}[1][]
{\begin{mainthm}[#1]}
	{\end{mainthm}}

\theoremstyle{definition}
\theoremstyle{remark}

\newtheorem{rem}[thm]{Remark}

\numberwithin{equation}{section}

\newcommand {\diam }{\mathrm{diam}}
\newcommand {\Isom }{\mathrm{Isom}}

\newcommand {\vol }{\mathrm{vol}}

\newcommand {\Z}{\mathbb{Z}}

\newcommand {\R}{\mathbb{R}}

\newcommand {\GH}{\xrightarrow{GH}}

\newcommand {\spa}[1]{\langle{#1}\rangle}

\newcommand {\lap }{\Delta}
\newcommand {\grad }{\triangledown}

\newcommand {\norm}[1]{\Vert{#1}\Vert}
\newcommand {\Ric}{\mathrm{Ric}}

\newcommand {\myarrow}[1]{\mathop{\longrightarrow}\limits^{#1}}

\newcommand{\XXint}[3]{{
		\setbox0=\hbox{$#1{#2#3}{\int}$}
		\vcenter{\hbox{$#2#3$}}\kern-.5\wd0}}

\begin{document}

\maketitle
\begin{abstract}

In \cite{NPZ24}, Navarro–Pan–Zhu proved that the fundamental group of an open manifold with nonnegative Ricci curvature and linear volume growth contains a subgroup isomorphic to $\mathbb{Z}^k$ with finite index. They further asked whether the existence of a torsion-free element in the fundamental group forces the universal cover to split off an isometric $\mathbb{R}$-factor (Question 1.3 of \cite{NPZ24}).

In this article, we provide an affirmative answer to this question. Specifically, we prove that if an open manifold with nonnegative Ricci curvature has linear volume growth, then its universal cover is isometric to a metric product $\mathbb{R}^k \times N$, where $N$ is an open manifold with linear volume growth and $k$ is the integer such that $\pi_1(M)$ contains a $\mathbb{Z}^k$-subgroup of finite index. As a direct consequence, if the Ricci curvature is positive at some point, then the fundamental group is finite.

We also establish that for an open manifold $M$ with nonnegative Ricci curvature, if the infimum of its volume growth order is strictly less than $3$ and $\tilde{M}$ has Euclidean volume growth, then the universal cover $\tilde{M}$ splits off an $\mathbb{R}^{n-2}$-factor. As an application, if $M$ has first Betti number $b_1 = n-2$ and $\tilde{M}$ has Euclidean volume growth, then its universal cover admits such a splitting. This result provides a partial answer to \cite[Question 1.6]{PY24}.

	

\end{abstract}

\setcounter{tocdepth}{1}
\tableofcontents

\section{Introduction}

The celebrated Cheeger--Gromoll's splitting theorem states that if an open manifold with nonnegative Ricci curvature contains a line, then it splits isometrically as metric product $\mathbb{R} \times N$, where $N$ also has nonnegative Ricci curvature \cite{CG71}. An important application of this result is the following structure theorem for compact manifolds.

\begin{thm}\label{split-thm}
	Let $M$ be a compact manifold with nonnegative Ricci curvature. Then its universal cover $\tilde{M}$ splits isometrically as $\mathbb{R}^k \times N$, where $k$ is the integer such that $\pi_1(M)$ contains a normal $\Z^k$ with finite index, and $N$ is a compact manifold.
\end{thm}

Since $N$ is compact, the deck transformation group $\pi_1(M)$ acts isometrically and separately on $\mathbb{R}^k \times N$. That is, each $\gamma \in \pi_1(M)$ decomposes as $\gamma(x, y) = (\gamma_X(x), \gamma_Y(y))$ for all $(x,y) \in \mathbb{R}^k \times N$, where $\gamma_X \in \operatorname{Isom}(\mathbb{R}^k)$ and $\gamma_Y \in \operatorname{Isom}(N)$.

The compactness assumption in Theorem~\ref{split-thm} cannot be removed. For example, Nabonnand \cite{Nab80} constructed a doubly warped metric with positive Ricci curvature on $\mathbb{R}^3 \times S^1$ that exhibits quadratic volume growth. The universal cover of such a manifold cannot split off an $\mathbb{R}$-factor isometrically, since any such splitting would imply the vanishing of the Ricci curvature in the $\mathbb{R}$-direction at every point.

It is well known that an open (i.e., complete and non-compact) manifold $M$ with nonnegative Ricci curvature has at least linear volume growth \cite{Cal75,Yau76}. This means that there exist a point $p \in M$ and a constant $c > 0$ such that for all $r > 0$, the volume of the $r$-geodesic ball satisfies
\begin{equation}\label{Def}
	\operatorname{vol}(B_r(p)) \ge c r.
\end{equation}

Hence, in terms of volume growth rate, the class of open manifolds with nonnegative Ricci curvature that most closely approximates the compact case (which itself exhibits zero-order growth) consists of those with \emph{linear volume growth}. More precisely, if the lower bound in \eqref{Def} is also sharp in the sense that there exists a constant $C > 0$ such that $\operatorname{vol}(B_r(p)) \le C r$ holds for all $r > 0$, then we say $M$ has linear volume growth.

This class of manifolds, studied by Sormani in the late 1990s (see, e.g., \cite{Sor98,Sor99,Sor00}), shares several topological properties with the compact case. For instance, Sormani \cite{Sor99} proved that the fundamental group of such a manifold is finitely generated. More recently, Navarro, Pan, and Zhu \cite{NPZ24} showed that this fundamental group contains a free abelian subgroup $\mathbb{Z}^k$ of finite index.

In this article, our first main result generalizes Theorem~\ref{split-thm} to the case of manifolds with linear volume growth.

\begin{maintheorem}\label{Fibration}
	Let $M$ be an open manifold with nonnegative Ricci curvature. If $M$ has linear volume growth, then the universal cover $\tilde M$ splits isometrically as $\R^k\times N$, where $k$ is the integer such that $\pi_1(M)$ contains a normal $\Z^k$ with finite index and $N$ has linear volume growth.

\end{maintheorem}

This theorem provides a positive answer to \cite[Question 1.3]{NPZ24}, which asked whether the presence of a torsion-free element in the fundamental group forces the universal cover to split off an $\mathbb{R}$-factor.

\begin{rem}\label{rem1}
Since the space $N$ in Theorem~\ref{Fibration} may contain a line, the separate action of $\pi_1(M)$ on the product $\mathbb{R}^k \times N$ is less obvious than in the compact case. Nevertheless, with a little extra work, we can describe this group action as follows.
\begin{itemize}
	\item [(1)] The $\mathbb{R}^k$-factor of $\tilde{M}$ can be chosen so that the action of $\pi_1(M)$ on the product $\mathbb{R}^k \times N$ is separate: the $\mathbb{Z}^k$-subgroup acts freely by translations on the $\mathbb{R}^k$-factor, while the closure of the induced $\pi_1(M)$-action on $N$ is compact. Consequently, $M$ admits a finite cover $\hat{M}$ that is the total space of a fiber bundle over the $k$-torus $T^k$ with fiber diffeomorphic to $N$.
	\item [(2)] There exists an isometric decomposition $\tilde{M} = \mathbb{R}^{b_1} \times \mathbb{R}^{k-b_1} \times N$ and a normal subgroup $K \subset \pi_1(M)$ such that $\pi_1(M)/K \cong \mathbb{Z}^{b_1}$, where $b_1$ is the first Betti number of $M$. The group $K$ acts trivially on the $\mathbb{R}^{b_1}$-factor, while $\pi_1(M)/K$ acts freely by translations on the $\mathbb{R}^{b_1}$-factor. Moreover, the quotient $(\mathbb{R}^{k-b_1} \times N)/K$ has linear volume growth. Consequently, $M$ admits a fiber bundle structure over the torus $T^{b_1}$ with fiber diffeomorphic to $(\mathbb{R}^{k-b_1} \times N)/K$.

\end{itemize}
\end{rem}

In \cite{NPZ24}, it was also shown that the fundamental group is finite when the Ricci curvature is positive. As a direct corollary of Theorem \ref{Fibration}, we are able to relax this assumption, making the result more analogous to the compact case.
\begin{cor}\label{Finiteness}
	Let $M$ be an open manifold with nonnegative Ricci curvature and linear volume growth. If the Ricci curvature is positive at some point, then $\pi_1(M)$ is finite.
\end{cor}

The second corollary provides a geometric refinement of the result in \cite{NPZ24} on the vanished escape rate of $\pi_1(M,p)$, showing that the minimal representing loops are, in fact, contained in a bounded ball.

\begin{cor}\label{Bounded}
	Let $(M, p)$ be an open manifold with nonnegative Ricci curvature and linear volume growth. Then there exists $r>0$, such that all minimal geodesic loops at $p$ representing elements of $\pi_1(M, p)$ lie within $B_r(p)$.
\end{cor}
In general, the combination of nonnegative Ricci curvature and linear volume growth is insufficient to guarantee that open manifolds of dimension at least~4 have finite topological type. Indeed, Menguy constructed examples of 4-dimensional open manifolds with positive Ricci curvature and bounded diameter growth—a condition that implies linear volume growth—yet which have infinite second Betti numbers \cite{Men00}. Nevertheless, our third corollary shows that in dimension $4$, an infinite fundamental group forces any manifold in this class to have finite topological type.

\begin{cor}\label{dim=4}
	Let $M$ be an open $4$-manifold with nonnegative Ricci curvature and linear volume growth. If $\pi_1(M)$ is infinite, then $M$ is homotopy equivalent to a closed submanifold.
\end{cor}

An $n$-manifold $N$ is said to have Euclidean volume growth if there exist a point $q \in N$ and a constant $v > 0$ such that for all $r > 0$,
\[
\vol(B_r(q)) \ge v r^n.
\]

In \cite{Ye24}, Ye proved that for an open manifold $M$ with nonnegative Ricci curvature, if the infimum of its volume growth order is strictly less than $2$, then the universal cover $\tilde{M}$ has Euclidean volume growth if and only if $M$ is flat. He also provided an example demonstrating that this bound is sharp, as the conclusion fails for some manifolds with volume growth order exactly $2$. Our second main result establishes the following theorem, which addresses the case where the infimum of the volume growth order is strictly less than $3$. It demonstrates that, under the additional assumption that the universal cover has Euclidean volume growth, a strong rigidity conclusion on the geometry and topology of $M$ can still be obtained.

\begin{maintheorem}\label{Quadra}
	Let $M$ be an open $n$-manifold with nonnegative Ricci curvature. 
	Suppose the following conditions hold:
	\begin{itemize}
		\item The infimum of the volume growth order of $M$ is strictly less than $3$; 
		that is, there exist $p \in M$ and $\delta \in (0,1)$ such that
		\[
		\liminf_{r \to \infty} \frac{\vol(B_r(p))}{r^{3-\delta}} < \infty.
		\]
		
		\item The universal cover $\tilde{M}$ has Eulcidean volume growth.
	\end{itemize}
	Then $\tilde{M}$ splits off an $\mathbb{R}^{n-2}$-factor isometrically. In particular, $M$ has nonnegative sectional curvature and its soul is a flat manifold of dimension $n-1$ or $n-2$.
\end{maintheorem}

A application of Theorem \ref{Quadra} yields the following result on the first Betti number.
\begin{cor}\label{b_1>n-3}
	Let $M$ be an open $n$-manifold with nonnegative Ricci curvature. 
	If the first Betti number $b_1(M) \geq n-2$, then $\pi_1(M)$ is finitely generated. 
	If we further assume that the universal cover $\tilde{M}$ has Euclidean volume growth, 
	then the conclusion of Theorem \ref{Quadra} holds.
\end{cor}
Corollary~\ref{b_1>n-3} provides a partial answer to \cite[Question 1.6]{PY24}, which conjectured that the universal cover of an open manifold with nonnegative Ricci curvature and first Betti number $n-2$ must split off an $\mathbb{R}^{n-2}$-factor. 

\begin{rem}
	In an early version of this paper, we established the following result: if $M$ is an open $n$-manifold with nonnegative Ricci curvature, $b_1(M) \ge n-2$, and $M$ has nondegenerate quadratic volume growth—that is, there exists a point $p \in M$ such that
	\[
	\limsup_{r \to \infty} \frac{\vol(B_r(p))}{r^2} \in (0, \infty),
	\]
	then the conclusion of Theorem~\ref{Quadra} holds. Indeed, these assumptions imply that the universal cover $\tilde{M}$ has Euclidean volume growth, making this result a corollary of Corollary~\ref{b_1>n-3}.
\end{rem}

Finally, we provide an effective version of our main theorems. Let $N$ be an open manifold with nonnegative Ricci curvature and $\Gamma$ a closed subgroup of its isometry group $\operatorname{Isom}(N)$. Denote by $\Omega(N)$ (respectively, $\Omega(N, \Gamma)$) the set of asymptotic cones (respectively, equivariant asymptotic cones) of $N$ (respectively, of the pair $(N, \Gamma)$).

\begin{prop}\label{EffectiveVer.}
	Let $M$ be an open $n$-manifold with nonnegative Ricci curvature, and let $(\hat{M}, \hat{p}) \to (M, p)$ be a normal covering with deck transformation group $\Gamma$. 
	Suppose that there exists an integer $k$ such that:
	\begin{itemize}
		\item For every $(Y, y) \in \Omega(\hat M)$, $Y$ has a pole at $y$, and the dimension of the maximal Euclidean factor of $Y$ equals that of the tangent cone of $Y$ at $y$, with both being at most $k$.
		\item There exists some $(X, x, G) \in \Omega(\hat M, \Gamma)$ such that $X$ splits off an $\mathbb{R}^k$-factor on which $G$ acts cocompactly.
	\end{itemize}
	Further, suppose that either one of the following conditions holds:
	\begin{itemize}
		\item[(1)] $M$ has the infimum of volume growth order $<2$; that is, there exists $\delta \in (0,1)$ such that $M$ satisfies
		\begin{equation}\label{1+alpha}
			\liminf_{r \to \infty} \frac{\vol(B_r(p))}{r^{2-\delta}} < \infty.
		\end{equation}
		\item[(2)] $\hat M$ has Euclidean volume growth and $k = n-2$.
	\end{itemize}
	Then $\Gamma$ contains a $\Z^{k}$-subgroup with finite index, and $(\hat{M}, \hat{p})$ admits an isometric splitting of the form $(\mathbb{R}^k \times N, (0^k, q))$. 
	In case (1), the manifold $N$ satisfies condition~\eqref{1+alpha} with the base point $p$ replaced by some $q \in N$.
\end{prop}

To conclude this section, we outline the main ideas behind our proofs. Under the assumptions of our theorems, we establish the stability of maximal Euclidean factors in every asymptotic cone of a certain covering space, employing distinct approaches for each result. For Theorem~\ref{Fibration}, the key ingredient is the structure theorem for equivariant asymptotic cones of normal covers of $M$ with linear volume growth \cite[Corollary 3.51]{NPZ24}. For Theorem~\ref{Quadra}, the main technical tools are Pan's critical scaling argument \cite{Pan18} and the virtual abelian property of the fundamental group under the stable Euclidean orbit condition \cite{Pan21}. By \cite[Theorem 1.10]{HH24}, this stability implies that on the relevant covering space, harmonic functions of almost linear growth in fact exhibit asymptotically linear growth, with their dimension determined by that of the maximal Euclidean factor in the asymptotic cones. In Theorem~\ref{Fibration}, the slow volume growth of the base space allows us to apply a Liouville-type theorem to the subharmonic function obtained from the norm-square of the gradients of these harmonic functions, demonstrating that it must be constant. The desired splitting then follows from a standard splitting criterion. In Theorem~\ref{Quadra}, we apply the slicing theorem of Cheeger and Naber \cite{CN15} to these harmonic functions on the covering spaces.

This paper is organized as follows. Section 2 presents several lemmas that will be used throughout the paper. Section 3 is devoted to establishing properties of harmonic functions with asymptotically linear growth. The proof of Proposition~\ref{EffectiveVer.}(1) and its consequences—including Theorem~\ref{Fibration} and several corollaries—is contained in Section 4. Finally, Section 5 addresses the proof of Proposition~\ref{EffectiveVer.}(2) along with Theorem~\ref{Quadra} and its associated corollaries.

\section{Some Preparations}

\subsection{Conventions}
\begin{itemize}
	\item The letter $d$ denotes the distance function on the relevant spaces; when there is no risk of ambiguity, we omit explicit reference to the underlying space.
	\item The constants $C$ and $c$ represent positive numbers whose values may vary from line to line. 
	\item The notation $\GH$ denotes Gromov-Hausdorff convergence or its equivariant version, as determined by the context. 
	\item We say that a space splits off an $\mathbb{R}^k$-factor to indicate an isometric splitting, unless otherwise specified.
\end{itemize}

\subsection{Linear volume growth}

The following result is essentially due to Sormani \cite{Sor98,Sor99}.

\begin{thm}\label{Isom}
	Let $M$ be an open manifold with nonnegative Ricci curvature and linear volume growth. Then one of the following holds.
	\begin{itemize}
		\item [(1)]$M$ contains no line and every asymptotic cone is $(\R_{\ge0},0)$ and the isometry group of $M$, $\Isom (M)$ is compact.
		\item [(2)] $M$ is isometric to $\R\times N$ for some compact $N$.
	\end{itemize}
	
\end{thm}

\begin{proof}
	We prove only the compactness of $\operatorname{Isom}(M)$ in case (1). Fix $p \in M$. If $\operatorname{Isom}(M)$ is not compact in case (1), then there exists a sequence $\{g_i\} \subset \operatorname{Isom}(M)$ such that $r_i := d(g_i(p), p) \to \infty$. After passing to a subsequence, we have the pointed Gromov-Hausdorff convergence
	\[
	(r_i^{-1}M, p, g_i) \xrightarrow{\mathrm{GH}} (\mathbb{R}_{\ge 0}, 0, g_\infty).
	\]
	It then follows that
	\[
	d(g_\infty(0), 0) = \lim_{i\to\infty} r_i^{-1}d(g_i(p), p) = 1,
	\]
	which leads to a contradiction.
\end{proof}

Essential to our proof of Theorem \ref{Fibration} is the following highly nontrivial theorem from \cite[Corollary 3.51]{NPZ24}.
\begin{thm}\label{NPZlem}
	Let $N$ be an open manifold with nonnegative Ricci curvature and linear volume growth. Let $\hat{N}$ be a normal cover of $N$ with deck transformation $\Gamma\cong\Z$. Then $\Omega(\hat{N},\Gamma)$ contains a unique element as one of the following:
	\begin{enumerate}
		\item $(\mathbb{R}^{2}, 0^2, \mathbb{R} \times \mathbb{Z}_{2})$, where $\Z_2$ fixes the origin $0^2$, 
		\item $(\mathbb{R}^{2}, 0^2, \mathbb{R})$,
		\item $(\mathbb{R} \times [0, \infty), (0,0), \mathbb{R})$.
	\end{enumerate}
\end{thm}

\subsection{Growth of orbits for the deck transformation group}
Let $(\hat N,\hat p)\to(N,p)$ be a Riemannian normal cover with deck transformation $\Gamma$. Define $$\Gamma(r):=\{\gamma\in\Gamma|d(\hat p,\gamma(\hat p))\le r\}.$$

The following lemma indicates that the growth of $|\Gamma(r)|$ (the number of elements in $\Gamma(r)$) is controlled by the volume ratio. We refer to \cite{An90-II} for the proof.

\begin{lem}\label{|Gamma(r)|}
	
	If $\Ric\ge0$, then
	$$\frac1{2^n}\frac{\vol(B_{r}(\hat p))}{\vol (B_{ r}(p))}\le|\Gamma(r)|\le2^n\frac{\vol(B_{r}(\hat p))}{\vol (B_r(p))}.$$ Consequently, for $R\ge r>0$, $$1\le \frac{|\Gamma(R)|}{|\Gamma(r)|}\le4^n \left(\frac{R}r\right)^n.$$
	
\end{lem}

For a group $\Gamma$ isomorphic to $\mathbb{Z}^s$, its orbit growth is at least of polynomial order $s$.

\begin{lem}\label{18:09}
	If $\mathrm{Ric} \ge 0$ and $\Gamma \cong \mathbb{Z}^s$, then there exist constants $c, C > 0$ such that for all sufficiently large $r > 0$,
	\[
	|\Gamma(r)| \ge c r^s \quad \text{and} \quad \vol(B_r(p)) \le C r^{n-s}.
	\]
\end{lem}
\begin{proof}
	Let $\{\gamma_1,...\gamma_{s}\}$ be a generator of $\Z^{s}$. Note that for any $\gamma=\gamma_1^{k_1}...\gamma_{s}^{k_{s}}$, $$\|\gamma\|\le k_1\|\gamma_1\|+...+k_{s}\|\gamma_{s}\|\le s\max\{|k_i|\}\max\{\|\gamma_i\|\},$$where $\|\gamma\|:=d(\gamma(\hat p),\hat p)$. This implies for any large $r$,
	\begin{equation*}
		|\Gamma(r)|\ge \left(\frac{r}{s\max\{\|\gamma_i\|\}}\right)^{s}.
	\end{equation*}
	
	The latter conclusion is by combining the above estimate with the volume comparison and Lemma \ref{|Gamma(r)|}.
\end{proof}

\subsection{Stability lemma for maximal orbits}

We now turn to the proof of Lemma~\ref{CriticalScaling}, which concerns a specific stability property of equivariant asymptotic cones. The argument commences by recalling two elementary facts in Euclidean geometry.

\begin{lem}\label{R^kaction}
	Let $G$ be a closed nilpotent subgroup of $\Isom (\R^k)$. If $\R^k/G$ is compact, then $G$ acts on $\R^k$ by translations. 
\end{lem}
\begin{proof}
	The center of $G$, denoted by $C(G)$, is a finite-index subgroup of $G$ (see \cite[Lemma 1.4]{Pan22-II} for a proof). It follows that the quotient $\mathbb{R}^k / C(G)$ is compact. A fundamental fact is that any abelian cocompact subgroup of $\operatorname{Isom}(\mathbb{R}^k)$ acts by translations on $\mathbb{R}^k$. Consequently, $C(G)$ acts on $\mathbb{R}^k$ by translations. Moreover, an elementary fact in linear algebra states that an isometry of $\R^k$ which commutes with $k$ linearly independent translations is itself a translation. Hence, it follows that $G$ also acts by translations.
\end{proof}

\begin{lem}\label{R^kcocompact}
	Let $K$ be a closed subgroup of $\operatorname{Isom}(\mathbb{R}^k)$. If the orbit $K(0^k)$ is not contained in any half-space, i.e., there exists no unit vector $v \in \mathbb{R}^k$ such that $v \cdot x \ge 0$ for all $x \in K(0^k)$, then the quotient space $\mathbb{R}^k / K$ is compact.
\end{lem}
\begin{proof}
	Argue by contradiction. If $\R^k/K$ is not compact. There exists a unit speed ray $\omega:[0,\infty)\to\R^k/K$ starting from $[0^k]$, where $[0^k]$ is the image of $0^k$. We lift $\omega$ to a ray $\bar\omega$ in $\R^k$ emanating from $0^k$. For any $s\in[0,\infty)$, $$s=\|\bar\omega(s)-0^k\|=d(\omega(s),[0^k])=d(\bar \omega(s),K(0^k)).$$That is, for any $x\in K(0^k)$ $$\|x-s\bar\omega(1)\|\ge s$$ for any $s$, which implies $K(0^k)$ is contained in the half space $\{x\in\R^k|x\cdot(-\bar\omega(1))\ge0\}$. This is a contradiction.
\end{proof}

\begin{lem}\label{CriticalScaling}
	
	Let $M$ be an open manifold with nonnegative Ricci curvature, and let $\Gamma$ be a closed nilpotent subgroup of $\operatorname{Isom}(M)$. Suppose there exists an integer $k$ such that:
	\begin{itemize}
		\item For every $(Y, y) \in \Omega(M)$, $Y$ has a pole at $y$, and the dimension of the maximal Euclidean factor of $Y$ equals that of the tangent cone of $Y$ at $y$, with both being at most $k$.
		
		\item There exists some $(X,x,G)\in\Omega(M,\Gamma)$ such that $X$ splits off an $\mathbb{R}^{k}$-factor on which $G$ acts cocompactly.
	\end{itemize}
	Then, for every $(Y, y, H) \in \Omega(M, \Gamma)$, the orbit $H(y)$ is isometric to $\R^k$.
\end{lem}
\begin{proof}

 For this $(X,x,G)$, a blowing-down argument yields that there exists $(X',x',G')\in\Omega(M,\Gamma)$ such that $X'$ splits off an $\mathbb{R}^{k}$-factor on which $G'$ acts transitively. To simplify notation, we continue to write $(X,x,G)$ for this new triple and assume that $G$ acts transitively on the $\mathbb{R}^{k}$-factor. 
	
	It suffices to prove the claim: for every $\epsilon>0$, $(Y, y, H) \in \Omega(M, \Gamma)$, there exists $(Z,z,K)\in\Omega(M,\Gamma)$ such that $K(z)$ is isometric to $\R^k$, and $$d_{GH}((H(y),y),(K(z),z))\le\epsilon.$$

	Fix $p\in M$. The proof is by a standard critical scaling argument which was first developed in \cite{Pan19}.
	
	Argue by contradiction. If there exist a small $\epsilon_0>0$ and $(Y,y,H)\in\Omega(M,\Gamma)$ such that 
	\begin{equation}\label{1337}
		d_{GH}((H(y),y),(K(z),z))> \epsilon_0
	\end{equation}
	for any $(Z,z,K)\in\Omega(M,\Gamma)$ satisfying $K(z)$ is isometric to $\R^k$.
	
	We choose $r_i\to \infty$ and $s_i\to\infty$ satisfying,
	\begin{itemize}
		\item $(r_i^{-1}M,p,\Gamma)\GH(X,x,G)$.
		\item $(s_i^{-1}M,p,\Gamma)\GH(Y,y,H)$.
		\item  $r_i^{-1}s_i\to\infty$.
	\end{itemize}
	Put $(M_i,p_i):=(s_i^{-1}M,p)$ and $\lambda_i:=r_i^{-1}s_i$. Then the above conditions are equivalent to,
	\begin{itemize}
		\item $(\lambda_iM_i,p_i,\Gamma)\GH(X,x,G)$.
		\item $(M_i,p_i,\Gamma)\GH(Y,y,H)$.
		\item $\lambda_i\to\infty$.
	\end{itemize}
	
	Define
	\begin{align*}
		A_i:=&\{t\in[1,\lambda_i]|d_{GH}((t\Gamma(p_i),p_i),(K(z),z))\le 0.1\epsilon_0,\\&\text{ for some }(Z,z,K)\in\Omega(M,\Gamma), \text{ and }K(z)\text{ is isometric to }\R^k\},
	\end{align*}
	where $t\Gamma(p_i)$ denotes the orbit $\Gamma(p_i)$ equipped with the extrinsic distance of $tM_i$. Obvious for all large $i$, $\lambda_i\in A_i$ and $1\notin A_i$. Choose $t_i\in A_i$ such that $\inf A_i\le t_i<\inf A_i+i^{-1}$.

	Now passing to a subsequence, we assume
	\begin{equation}\label{1:52}
		(t_iM_i,p_i,\Gamma)\GH(Y_1,y_1,H_1).
	\end{equation} 
	By the choice of $t_i$ and $A_i$, there exists $(Z,z,K)\in\Omega(M,\Gamma)$ satisfying $K(z)$ is isometric to $\R^k$ and 
	
	\begin{equation}\label{Orbitclosedness}
		d_{GH}((H_1(y_1),y_1),(K(z),z))\le0.1\epsilon_0.
	\end{equation}
	
	Under the given conditions and by Cheeger-Colding theory, $(Z, z)$ splits off a unique $\mathbb{R}^k$-factor, on which the group $K$ acts transitively. Assume $(Y_1,y_1)$ splits off a unique $\R^{k'}$-factor, where $k'\le k$. Since $y_1$ is a pole of $Y_1$ and its tangent cone $T_{y_1}Y_1$ contains at most an $\mathbb{R}^{k'}$-factor, it follows that $H_1(y_1)$ is contained in the $\mathbb{R}^{k'}$-slice passing through $y_1$. From this fact and (\ref{Orbitclosedness}), we conclude that $k'=k$. It follows from (\ref{Orbitclosedness}) and Lemma~\ref{R^kcocompact} that the induced $H_1$-action on the $\R^k$-factor of $Y_1$ is cocompact. Since $\Gamma$ is assumed to be nilpotent, $H_1$ is also nilpotent. By Lemma~\ref{R^kaction}, this implies that $H_1$ acts cocompactly by translations on $\mathbb{R}^k$. So $H_1(y_1)$ is isometric to some subset $\R^l\times\Z^{k-l}$ of $\R^k$. Therefore, there exists a small $\delta\in(0,1)$ such that
	\begin{equation*}\label{0.05epsilon}
		d_{GH}((\delta H_1(y_1),y_1),(K(z),z))\le0.05\epsilon_0.
	\end{equation*}
	
	Consider the convergence, $$(\delta t_iM_i,p_i,\Gamma)\GH (\delta Y_1,y_1,H_1).$$
	If $\delta t_i\ge 1$, then the above convergence implies $\delta t_i\in A_i$ for all large $i$, which contradicts the choice of $t_i$. Therefore $t_i\in[1,\delta^{-1})$. Up to a subsequence, we assume $t_i\to C\in[1,\delta^{-1}]$. So convergence (\ref{1:52}) becomes$$(t_iM_i,p_i,\Gamma)\GH (CY,y,H)$$ which implies $(Y,y,H)=(C^{-1}Y_1,y_1,H_1)$. By (\ref{Orbitclosedness}), $$d_{GH}((H(y),y),(K(z),z))=d_{GH}((C^{-1}H_1(y_1),y_1),(K(z),z))\le 0.5\epsilon_0,$$which contradicts to the contradicting assumption (\ref{1337}).
	
\end{proof}

\subsection{Virtual $\mathbb{Z}^k$-subgroup under stable Euclidean orbit condition}

We also need the following theorem, based on the results in \cite[Theorem A, Theorem 0.2]{Pan21}.

\begin{thm}\label{escaprate}
	Let $\hat{M}\to M$ be a normal cover with deck transformation group $\Gamma$, and suppose $M$ has nonnegative Ricci curvature. If for every $(Y, y, G) \in \Omega(\hat{M}, \Gamma)$, the orbit $G(y)$ is geodesic in $Y$ and isometric to a standard Euclidean space $\R^k$, then $\Gamma$ contains a $\Z^k$-subgroup of finite index.
\end{thm}

Although the original version of \cite[Theorem A, Theorem 0.2]{Pan21} is stated for the universal cover, this assumption is not essential. We therefore adapt the theorem to the case of a normal cover. Furthermore, since the relationship between $\dim G(y)$ and the maximal rank of free abelian subgroups of $\Gamma$ is not explicitly addressed in \cite{Pan21}, we provide a proof of this fact, based on the aforementioned results.
\begin{proof}
	
	We only prove the relationship between $\dim G(y)$ and the maximal rank of free abelian subgroups of $\Gamma$ as stated in Theorem~\ref{escaprate}. Since $\Gamma$ has escape rate $0$ by \cite[Theorem 0.2]{Pan21}, it follows from \cite{Sor99} that $\Gamma$ is finitely generated. Applying \cite{KW11}, we deduce that $\Gamma$ contains a torsion-free nilpotent subgroup of finite index. By \cite[Theorem A]{Pan21}, this subgroup is virtually abelian. Since a torsion-free nilpotent group that is virtually abelian is necessarily abelian, it follows that $\Gamma$ contains a subgroup isomorphic to $\mathbb{Z}^k$ with finite index.

	We now assume that for every $(Y, y, G) \in \Omega(\hat{M}, \Gamma)$, the orbit $G(y)$ is geodesic in $Y$ and isometric to a standard Euclidean space $\mathbb{R}^s$ for some $s$. Our goal is to show that $s = k$.

	Note that without loss of generality, we may assume $\Gamma \cong \mathbb{Z}^k$. This assumption is justified by the following: the $\mathbb{Z}^k$-subgroup of $\Gamma$ (which has finite index) has a limit group in the asymptotic cone that also has finite index in the limit of $\Gamma$. Consequently, the orbit of this $\mathbb{Z}^k$-limit group at the reference point is isometric to the orbit of the $\Gamma$-limit group.
	
	We proceed by induction on $k$. The case $k = 0$ is trivial. We now assume the result holds for all nonnegative integers $\le k$. By the above assumption, for every $(Y, y, G) \in \Omega(\hat{M}, \Gamma)$, the group $G$ is abelian and acts transitively on $G(y) = \mathbb{R}^s$. Consequently, $G$ acts by translations on $\mathbb{R}^s$. Let $\hat\Gamma$ be a $\Z^{k-1}$ subgroup of $\Gamma$ such that $\Gamma/\hat\Gamma\cong\Z$. Fix $\hat p\in\hat M$ and put $(\check M,\check p)=(\hat M,\hat p)/\hat\Gamma$. Given a sequence $r_i\to\infty$, passing to a subsequence, we have,

	$$
	\xymatrix{
		(r_i^{-1}\hat M,\hat p,\hat\Gamma,\Gamma) \ar[rr]^{GH}\ar[d]_{}&&(Y,y,\hat G,G) \ar[d]^{} \\
		(r_i^{-1}\check M,\check p,\Gamma/\hat\Gamma)\ar[rr]^{GH}&  & (X,x,H),}
	$$
	where $G(y)$ is isometric to $\R^s$. Since $\Gamma$ has escape rate $0$, so does the quotient group $\Gamma / \hat{\Gamma}$. Applying \cite[Proposition 2.27]{NPZ24}, we conclude that $H(x)$ is homeomorphic to $\mathbb{R}$. Since $H(x)$ is also homeomorphic to the quotient space $G(y)/\hat{G}$, it follows that $G(y)/\hat{G}$ is homeomorphic to $\mathbb{R}$. 
	
	From the previous discussion, $\hat{G}$ acts on $G(y) = \mathbb{R}^s$ by translations; consequently, $\hat{G}(y)$ must be a linear subspace isomorphic to $\mathbb{R}^{s-1}$. Due to the arbitrary choice of the sequence $r_i$, we obtain that for every $(Y, y, \hat{G}) \in \Omega(\hat{M}, \hat{\Gamma})$, the orbit $\hat{G}(y)$ is geodesic and isometric to $\mathbb{R}^{s-1}$. By the induction hypothesis, we have $s-1 = k-1$, and thus $s = k$. This completes the induction step.

\end{proof}

\section{Harmonic Functions with Asymptotically Linear Growth}

In this section we establish the following technical lemma, which is fundamental to our argument.

\begin{lem}\label{Key}
	Let $(N,p)$ be an open manifold with nonnegative Ricci curvature. Let $(\hat N,\hat p)\to (N,p)$ be a normal cover with deck transformation $\Gamma\cong\Z$. Suppose that for any $(X,x,G)\in\Omega(\hat N,\Gamma)$, $X$ splits off a unique $\R$-factor isometrically and the orbit $G(x)$ coincides with the line passing $x$. Then there exists a non-constant harmonic function $u:(\hat N,\hat p)\to(\R,0)$ satisfying that, 
	\customitemize{p}
	\item\label{p1} For any $\epsilon>0$, $\lim_{r\to\infty}\frac{\sup_{\partial B_r(\hat p)}|u|}{r^{1+\epsilon}}=0$.
	
	\item\label{p2}  For any $\gamma\in\Gamma\setminus\{e\}$, there exists a constant $b\in\R\setminus\{0\}$, such that $u\circ\gamma=u+b$.
	
\end{enumerate}

\end{lem}
\begin{proof}

Let $H^{1+\epsilon}(\hat N)$ be the space of harmonic functions $u:\hat N\to\R$ satisfying, there exists $C>0$ such that $$|u(x)|\le Cd(x,\hat p)^{1+\epsilon}+C$$ for any $x\in\hat N$.

By \cite[Theorem 1.10]{HH24}, $V:=\cap_{\epsilon>0} H^{1+\epsilon}(\hat N)$ has dimension $2$. Let $u\in V$ be a non-constant harmonic function with $u(\hat p)=0$. Fixing a generator $\gamma$ of $\Gamma$, since $\dim V=2$, there exist $a,b\in\R$ such that 
\begin{equation}\label{au+b}
	u\circ\gamma=au+b.
\end{equation}

Fix arbitrary $\epsilon_i\to 0^+$ as $i\to\infty$. Then	
\begin{equation}\label{WeakLinearGrowth}
	\lim_{r\to\infty}\frac{\|u\|_{C^0(B_r(\hat p))}}{r^{1+\epsilon_i}}=0.
\end{equation}

Let $C_i>0$ be the minimum number such that $\|u\|_{C^0(B_r(\hat p))}\le C_ir^{1+\epsilon_i}$ for any $r\ge 1$. We first assume 
\begin{equation}\label{C_k}
	\lim_{i\to\infty}C_i=\infty.
\end{equation}

Let $r_i\ge1$ be the minimum number such that $\|u\|_{C^0(B_{r_i}(\hat p))}= C_i{r_i}^{1+\epsilon_i}$. If $\{r_i|i=1,2,...\}$ is bounded, then up to a subsequence, we may assume $r_i\to r\ge1$. Consequently, $$C_i=\frac{\|u\|_{C^0(B_{r_i}(\hat p))}}{r_i^{1+\epsilon_i}}\to \frac{\|u\|_{C^0(B_{r}(\hat p))}}{r},$$which contradicts (\ref{C_k}). Thus passing to a subsequence, we assume $r_i\to\infty$. Let $x_i\in\partial B_{r_i}(\hat p)$ such that $|u(x_i)|=C_ir_{i}^{1+\epsilon_i}$.

By (\ref{au+b}) and (\ref{WeakLinearGrowth}),

$$0=\lim_{j\to\infty}\frac{|u(\gamma^j(\hat p))|}{d(\hat p,\gamma^j(\hat p))^{1+\epsilon_i}}\ge \lim_{j\to\infty}\frac{|(1+a+...+a^{j-1})b|}{j^{1+\epsilon_i}d(\hat p,\gamma(\hat p))^{1+\epsilon_i}},$$which implies
\begin{equation}\label{ab}
	|a|\le 1,\text{ or, }b=0.
\end{equation}

Define $u_i:=\frac{u}{C_ir_i^{1+\epsilon_i}}$. Then $u_i$ is harmonic satisfying $\|u_i\|_{C^0(B_r(\hat p,r_i^{-1}\hat N))}\le r^{1+\epsilon_i}$ for any $r\ge\frac1{r_i}$, and $\|u_i\|_{C^0(B_{1}(\hat p,r_i^{-1}\hat N))}=|u_i(x_i)|=1$. Now passing to a subsequence, we may assume
\begin{equation}\label{Con2}
	(r_{i}^{-1}\hat N,\hat p,x_{i},\Gamma,u_{i})\GH(\R\times Y,p_\infty,x_\infty,G,u_\infty),
\end{equation}
where $u_{i}\to u_\infty$ in $W^{1,2}$-sense, and $Y$ splits no line. Specially, $u_\infty$ is a linearly growth harmonic function satisfying $u_\infty(p_\infty)=0$, $\|u_\infty\|_{C^0(B_r(p_\infty))}\le r$ for any $r>0$, and $u_\infty(x_\infty)=1$. By \cite[Theorem 3.8]{HH24}, this implies that $u_\infty$ is the standard coordinate function of the $\R$-factor. Hence $x_\infty\in\partial B_1(p_\infty)\cap u_{\infty}^{-1}(1)$, which concludes that $x_\infty\in G(p_\infty)$. So there exists $\gamma_\infty\in G$ such that $\gamma_\infty(p_\infty)=x_\infty$. We choose $s_i\in\Z$ such that $\gamma^{s_i}\to\gamma_\infty$ with respect to (\ref{Con2}). By replacing $\gamma$ with its inverse if necessary and passing to a subsequence, we may assume without loss of generality that $s_i\to+\infty$. Then 
\begin{equation*}\label{23:30}
	u_{i}(\gamma^{s_i}(\hat p))\to u_\infty(x_\infty)=1.
\end{equation*}
That is, 
\begin{equation*}\label{23:31}
	\frac{|(1+a+...+a^{s_i-1})b|}{C_{i}r_{i}^{1+\epsilon_{i}}}=\frac{|u(\gamma^{s_i}(\hat p))|}{C_{i}r_{i}^{1+\epsilon_{i}}}\to1.
\end{equation*} 
Combining with (\ref{ab}), the above convergence implies $a=1$ and $b\neq0$. Hence we have 
\begin{equation*}\label{key}
	u\circ\gamma=u+b,\,b\neq0.
\end{equation*}

If (\ref{C_k}) does not hold, combining with Cheng-Yau gradient estimate, then for any $r>1$, $$\|\grad u\|_{C^0(B_{\frac12r}(\hat p))}\le C(n)\frac{\|u\|_{C^0(B_r(\hat p))}}{r}\le C(n) \liminf_{i\to\infty}C_i<\infty.$$ Up to a normalization, we assume $\sup_{N}\norm{\grad u}^2=1$. Note that $\gamma$ is an isometry, so
\begin{equation*}
	a=\sup_N\norm{\grad (au+b)}=\sup_N\norm{\grad (u\circ\gamma)}=\sup_N\norm{\grad u}=1.
\end{equation*}
A standard blowing down argument gives $b\neq0$ similar to above.

\end{proof}

\section{Splitting and Linear Volume Growth}
\subsection{Proof of Proposition \ref{EffectiveVer.}(1)}
Combining Lemma~\ref{CriticalScaling} and Theorem~\ref{escaprate}, we obtain the following conclusions:
\begin{itemize}

\item $\Gamma$ contains a subgroup isomorphic to $\mathbb{Z}^k$ with finite index.
	\item For every $(X, x, G) \in \Omega(\hat{M}, \Gamma)$, $(X, x)$ splits isometrically as $(\mathbb{R}^k \times Y, (0^k, y^*))$, where $Y$ contains no line, and $G(x) = \mathbb{R}^k \times \{y^*\}$.
\end{itemize}
Hence up to a finite cover, we assume $\Gamma\cong\Z^k$.

The proof proceeds by induction on $k$. The base case $k=0$ is trivial. For the inductive step, we assume the result holds for all nonnegative integers up to $k-1$. 

Choose a $\Z$-subgroup $\Lambda\subset\Gamma$ such that $\Gamma/\Lambda$ is isomorphic to $\Z^{k-1}$. Put $(\check M,\check p)=(\hat M,\hat p)/\Lambda$. Given the known conditions, for any sequence $r_i \to 0$, we may pass to a subsequence (still denoted by $r_i$) such that the following diagram commutes,

\begin{equation}\label{Dia1008}
	\xymatrix{
		(r_i^{-1}\hat M,\hat p,\Lambda,\Gamma) \ar[rr]^{GH}\ar[d]_{}&&(\R^{k}\times Y,(0^k,y^*),H,G) \ar[d]^{} \\
		(r_i^{-1}\check M,\check p,\Gamma/\Lambda)\ar[rr]^{GH}&  & (Z,z^*,K),}
\end{equation}
where $Y$ contains no line.

By the work of Cheeger--Colding--Naber \cite{CC97,CN12}, every limit group in this setting is a Lie group. Since $Y$ contains no line, the group $G$ acts separately on $\mathbb{R}^k$ and $Y$. Furthermore, as $G$ acts transitively on $\mathbb{R}^k$ and is abelian, it follows that the action on $\mathbb{R}^k$ is by translations. A direct verification shows that $K(z^*)$ is homeomorphic to the quotient $G((0^k, y^*)) / H = \mathbb{R}^k / H$. By \cite[Corollary 3.2]{PY24}, the subgroup $H$ contains a closed $\mathbb{R}$-subgroup, and $K$ contains a closed $\mathbb{R}^{k-1}$-subgroup. Therefore, we have the topological dimension estimates:
$$
k - 1 \le \dim K(z^*) = \dim(\mathbb{R}^k / H) \le k - 1,
$$
which forces $\dim K(z^*) = k - 1$ and $\dim(H((0^k, y^*))) = 1$. Note that this conclusion holds for every convergent sequence satisfying diagram (\ref{Dia1008}). In particular, $H((0^k, y^*))$ is a one-dimensional linear subspace of $\mathbb{R}^k$.

Therefore, the orthogonal complement of $H((0^k, y^*))$, denoted by $\mathbb{R}^{k-1} \times \{y^*\}$, is projected to a totally geodesic subspace of $Z$. Applying the splitting theorem, we conclude that $(Z, z^*)$ is isometric to $(\mathbb{R}^{k-1} \times Z_1, (0^{k-1}, z_1^*))$, where $Z_1$ contains no line and $K(z^*) = \mathbb{R}^{k-1} \times \{z_1^*\}$. Due to the arbitrariness of the sequence $r_i$, the normal covering $(\check{M}, \check{p}) \to (M, p)$ satisfies the hypotheses of Proposition~\ref{EffectiveVer.} for parameter $k-1$. By the inductive hypothesis, $(\check{M}, \check{p})$ admits an isometric splitting of the form $(\mathbb{R}^{k-1} \times \check{N}, (0^{k-1}, \check{q}))$, where $\check{ N}$ satisfies
\begin{equation}\label{AAA}
	\liminf_{r\to\infty}\frac{\vol(B_r(\check q))}{r^{2-\delta}}<\infty.
\end{equation}

We now lift the $\mathbb{R}^{k-1}$-factor from the splitting $\check M=\mathbb{R}^{k-1} \times \check{N}$ to $\hat{M}$ via the covering map. This yields an isometric splitting $(\hat{M}, \hat{p}) \cong (\mathbb{R}^{k-1} \times \hat{N}, (0^{k-1}, \hat{q}))$, where the $\Lambda$-action is trivial on the $\mathbb{R}^{k-1}$-factor and $\check{N} = \hat{N}/\Lambda$.

From the commutative diagram (\ref{Dia1008}), we observe that every triple $(X, x, H') \in \Omega(\hat{N}, \Lambda)$ satisfies the property that $X$ splits off a unique $\mathbb{R}$-factor and $H'(x)$ is the line through $x$. Therefore, by applying Lemma~\ref{Key} to the covering $(\hat{N}, \hat{q}) \to (\check{N}, \check{q})$, we conclude that there exists a nonconstant harmonic function $u: \hat{N} \to \mathbb{R}$ satisfying properties (\ref{p1}) and (\ref{p2}).

By (\ref{p2}), $\Lambda$-action preserves $\norm{\grad u}^2$. So we obtain a function $f:\check N\to\R_{\ge0}$ such that $f\circ\pi=\norm{\grad u}^2$, where $\pi$ is the covering map $\hat N\to \check N$. By the Bochner formula, $\lap f\ge0$. Combining (\ref{p1}), the maximal principle, and the Cheng-Yau gradient estimate, we conclude, for any $\epsilon>0$, any any large $r>r_\epsilon$,
\begin{equation*}
	\sup_{B_r(\check q)}f=\sup_{\partial B_r(\hat q)}\norm{\grad u}^2\le C(n)r^{-2}\sup_{\partial B_{2r}(\hat q)}|u|^2\le  C(n) r^{2\epsilon}.
\end{equation*}
By combining the above estimate with the condition (\ref{AAA}), it implies that, there exist $C>0$, and a sequence $r_i\to\infty$, such that
$$\frac{1}{r_i^2}\int_{B_{r_i}(\check q)}f^2\le \frac{\vol(B_{r_i}(\check q))}{r_i^2}\sup_{B_{r_i}(\check q)} f^2\le CC(n)r_i^{-\delta+4\epsilon}.$$
Applying \cite[Theorem A]{Kar82}, if $f$ is not constant, then the left hand side of the above inequality tends to $\infty$ as $r_i\to\infty$. This is a contradiction if we choose $\epsilon\le\frac\delta4$. In conclusion, we conclude that $\norm{\grad u}^2=f\circ\pi$ is constant.

Appealing to the Bochner formula, we conclude that $u$ is a nonconstant linear function on $\hat{N}$. This implies that $(\hat{N},\hat q)$ splits isometrically as $(\mathbb{R} \times N,(0,q))$ for some complete manifold $N$, and $u$ corresponds to a coordinate function along the $\mathbb{R}$-factor. Furthermore, from diagram (\ref{Dia1008}) we deduce that $N$ contains no line. It follows that the group $\Lambda$ acts on $\mathbb{R} \times N$ separately. By property~(\ref{p2}), the group $\Lambda$ acts effectively on $\mathbb{R}$ by translations. Consequently, for any $v \in \mathbb{R}$, the intersection $\Lambda(\hat{q}) \cap (\{v\} \times N)$ contains at most one element. This implies the existence of a constant $C > 0$ such that
\begin{equation}\label{Lambda}
	|\Lambda(r)| \le Cr.
\end{equation}

Now observe that $(-r, r) \times B_r(q) \subset B_{2r}(\hat{q})$, from which we derive
\[
\vol(B_r(q)) \le \frac{\vol(B_{2r}(\hat{q}))}{2r} \le \frac{C(n)|\Lambda(2r)|\vol(B_{2r}(\check{q}))}{r} \le C(n)C\vol(B_{2r}(\check q)),
\]
where the second inequality follows from Lemma~\ref{|Gamma(r)|} and the third from inequality~(\ref{Lambda}). Combining this with (\ref{AAA}), we conclude that  $$	\liminf_{r\to\infty}\frac{\vol(B_r( q))}{r^{2-\delta}}<\infty.$$

The inductive step is now complete.

\subsection{Proof of Theorem \ref{Fibration} and its corollaries}

\begin{lem}\label{Nonsplittingreduction}
	Let $\hat{N} \to N$ be as in Theorem~\ref{NPZlem}. If case (1) of Theorem~\ref{NPZlem} holds, then $\hat N$ is isometric to the product of $\R^2$ and a compact manifold. 
\end{lem}
\begin{proof}

Let $2\Gamma$ denote the index-$2$ subgroup of $\Gamma$. For any $r_i\to\infty$, up to a subsequence, we assume
\begin{equation}\label{Blowdown}
	(r_i^{-1}\hat N,\hat p,2\Gamma,\Gamma)\GH(\R^2,0^2,G_1,G),
\end{equation}
where $G\cong\R\times\Z/2\Z$. Observe that by \cite[Corollary 3.2]{PY24}, $G_1\subset G$ contains a closed $\R$-subgroup.

We claim $G_1$ does not contain the reflection $\mathbf{r}\in \Z/2\Z$. Otherwise, letting $\gamma\in\Gamma$ be a generator, there exist $\gamma^{2s_i}\to \mathbf{r}$ with respect to convergence (\ref{Blowdown}) for some integers $s_i$. Observe that $l_i:=r_i^{-1}d_{\hat N}(\gamma^{s_i}(\hat p),\hat p)$ are uniformly bounded; if not, blowing down $r_i^{-1}\hat N$ by $l_i^{-1}$ gives the same equivariant asymptotic cone, and $l_i^{-1}r_{i}^{-1}d_{\hat N}(\gamma^{2s_i}(\hat p),\hat p)\to 2$, which is a contradiction. So up to a subsequence, we may assume $\Gamma\ni \gamma^{s_i}\to g\in G$ with respect to (\ref{Blowdown}). Hence $g^2=\mathbf{r}$ which is impossible.

So $G_1$ is the translation $\R$-subgroup. By the arbitrariness of $r_i$, we conclude that any asymptotic cone of $\hat N/2\Gamma$ is isometric to $\R^1$. This implies $\hat N/2\Gamma$ is isometric to $\R\times S$ for some closed manifold $S$ (see \cite[Proposition 3.3]{Ye24} for a proof).

By lifting the $\mathbb{R}$-factor of $\hat{N}/2\Gamma = \mathbb{R} \times S$, we conclude that $\hat{N}$ is isometric to $\mathbb{R} \times \hat{S}$, where $2\Gamma$ acts trivially on the $\mathbb{R}$-factor. Furthermore, $\hat{S} \to S$ is a normal covering with deck transformation group $2\Gamma$. Applying the Cheeger--Gromoll's trick, we find that $\hat{S}$ is isometric to $\mathbb{R} \times \check{S}$. Note that $\check{S}$ must be compact; otherwise, the asymptotic cone of $\hat{N}$ would split off an $\mathbb{R}^3$-factor.
\end{proof}

\begin{lem}\label{SplittingInduction}
Let $\hat{N} \to N$ be as in Theorem~\ref{NPZlem}. Then $\hat N$ is isometric to $\R\times N_1$ for some open manifold $N_1$ with linear volume growth. 
\end{lem}
\begin{proof}
	If case (1) of Theorem~\ref{NPZlem} occurs, then Lemma~\ref{Nonsplittingreduction} yields the required conclusion.
	
	If case (2) of Theorem~\ref{NPZlem} occurs, then $N$ has a unique asymptotic cone $(\mathbb{R}, 0)$. By Theorem~\ref{Isom}, $N$ splits off an $\mathbb{R}$-factor, and so does $\hat{N}$. Applying Cheeger–Gromoll's trick once more, we conclude that $\hat{N}$ is isometric to $\mathbb{R}^2 \times K$ for some compact space $K$.
	
	If case (3) of Theorem~\ref{NPZlem} occurs, then Proposition~\ref{EffectiveVer.} applies to the covering space $(\hat{N}, \hat{p}) \to (N, p)$ for the case $k = 1$ and $\alpha = 0$, yielding the desired conclusion.
	
\end{proof}

\begin{lem}\label{Splitting&LVG}
	
	Let $\hat{M} \to M$ be a normal covering such that $M$ has nonnegative Ricci curvature and the deck transformation group $\Gamma$ is isomorphic to $\mathbb{Z}^k$. If $M$ has linear volume growth, then $\hat M$ splits isometrically as $\R^k\times N$, where $N$ has linear volume growth.
\end{lem}
\begin{proof}
	We proceed by induction on $k$. The case $k = 0$ is trivial. For the inductive hypothesis, we assume the validity of the theorem for all nonnegative integers up to $k-1$. 
	
	We choose a $\Z^{k-1}$-subgroup $\Lambda\subset \Gamma$ such that $\Gamma/\Lambda\cong\Z$. Then Lemma \ref{SplittingInduction} applies to $\hat M/\Lambda\to \hat M/\Gamma$ which yields that $\hat M/\Lambda$ is isometric to $\R\times N_1$ for some $N_1$ with linear volume growth. We lift the $\mathbb{R}$-factor from $\mathbb{R} \times N_1$ via the covering map $\hat{M} \to \hat M/\Lambda=\mathbb{R} \times N_1$, obtaining an isometric splitting $\hat{M} \cong \mathbb{R} \times \hat{M}_1$, where $\Lambda$ acts trivially on the $\R$-factor of $\R\times\hat M_1$. Now we can apply the inductive hypothesis to the covering $\hat M_1\to N_1 $ to conclude that $\hat M_1$ splits isometrically as $\R^{k-1}\times N$ where $N$ has linear volume growth.

\end{proof}

\begin{proof}[\textbf{Proof of Theorem \ref{Fibration} and Remark \ref{rem1}} ]

Let $\Gamma$ be a normal $\Z^k$-subgroup of $\pi_1(M)$ with finite index. Note that $\tilde M/\Gamma$ is a finite cover of $M$ which has linear volume growth. We apply Lemma \ref{Splitting&LVG} to $\tilde M\to\tilde M/\Gamma$, yielding that $\tilde M$ splits isometrically as $\R^{k}\times N$ with $N$ linear volume growth. This concludes the proof of Theorem \ref{Fibration}. We next address the action of $\pi_1(M)$ on $\tilde{M}$ that was mentioned in Remark~\ref{rem1}.

(1) If $N$ contains no line, then $\pi_1(M)$ acts separately on $\mathbb{R}^k$ and $N$. By Theorem~\ref{Isom}, the isometry group $\operatorname{Isom}(N)$ is compact. So the closure of induced $\pi_1(M)$-action on $N$ is compact. For any $\gamma\in\Gamma\setminus\{e\}$, if $\gamma$ acts on $\R^k$ trivially, then the subgroup generated by $\gamma$, $\spa{\gamma}$, acts on $N$ properly discontinuously, which contradicts that $\Isom (N)$ is compact. Consequently, the induced action of $\Gamma$ on $\mathbb{R}^k$ is effective and must be by translations. 
	
	We now turn to the case where $N$ contains a line. By Theorem~\ref{Isom}, $N$ is isometric to $\mathbb{R} \times S$ for some compact manifold $S$, and $\pi_1(M)$ acts separately on $\tilde M=\mathbb{R}^{k+1}\times S$, with the $\Gamma$-action on $\mathbb{R}^{k+1}$ being free. The analysis of this case is a direct consequence of the following elementary fact about isometry groups of Euclidean spaces.
	
	\begin{slem}\label{sublem}
		Let $G$ be a group acting isometrically on $\mathbb{R}^n$, and suppose $G$ contains a normal subgroup $\Gamma \cong \mathbb{Z}^k$ that is closed in $\operatorname{Isom}(\mathbb{R}^n)$. Then $\mathbb{R}^n$ admits a decomposition $\mathbb{R}^k \times \mathbb{R}^{n-k}$ such that:
		\begin{enumerate}
			\item $G$ acts separately on the two factors;
			\item $\Gamma$ acts freely by translations on the $\mathbb{R}^k$-factor;
			\item $\Gamma$ fixes a point in the $\mathbb{R}^{n-k}$-factor.
		\end{enumerate}
	\end{slem}
	
	The proof of this sublemma is provided in the appendix. Applying the sublemma to the case where $n = k+1$ and $G = \pi_1(M)$, we conclude that the $\mathbb{R}^{k+1}$-factor of $\tilde{M}$ admits a decomposition $\mathbb{R}^k \times \mathbb{R}^1$ such that $\pi_1(M)$ acts separately on the two factors: $\Gamma$ acts freely by translations on the $\mathbb{R}^k$-factor and fixes a point on the $\mathbb{R}^1$-factor. Since $\pi_1(M)/\Gamma$ is finite, the induced $\pi_1(M)$-action on $\mathbb{R}^1$ has compact closure. Therefore, $\tilde{M}$ admits an isometric decomposition
	\[
	\mathbb{R}^k \times N := \mathbb{R}^k \times (\mathbb{R}^1 \times S)
	\]
	satisfying the required conclusions.

(2) Note that the abelianization of $\pi_1(M)$ satisfies $\pi_1(M)/[\pi_1(M), \pi_1(M)] \cong \mathbb{Z}^{b_1} \oplus T$, where $b_1$ is the first Betti number of $M$ and $T$ is a finite abelian group. Let $K$ be the kernel of the natural homomorphism 
\[
\pi_1(M) \to \left( \pi_1(M)/[\pi_1(M), \pi_1(M)] \right) / T \cong \mathbb{Z}^{b_1}.
\]
Applying Lemma~\ref{Splitting&LVG} to the covering $\tilde{M}/K \to M$, we obtain an isometric splitting $\tilde{M}/K \cong \mathbb{R}^{b_1} \times N_2$, where $N_2$ has linear volume growth. By an argument analogous to that in part (1), and after possibly rechoosing the $\mathbb{R}^{b_1}$-factor if necessary, the quotient group $\pi_1(M)/K$ acts freely on $\mathbb{R}^{b_1}$ by translations. Lifting the $\mathbb{R}^{b_1}$-factor through the covering map $\tilde{M} \to \tilde{M}/K$, we deduce an isometric decomposition $\tilde{M} \cong \mathbb{R}^{b_1} \times \mathbb{R}^{k-b_1} \times N$. Consequently, the quotient $(\mathbb{R}^{k-b_1} \times N)/K$, which is isometric to $N_2$, inherits linear volume growth.

\end{proof}

\begin{proof}[\textbf{Proof of Corollary \ref{Finiteness}}]
	Let $(\tilde{M}, \tilde{p}) \to (M, p)$ be the universal cover. By Theorem~\ref{Fibration}, $\tilde{M}$ splits isometrically as $\mathbb{R}^k \times N$, where $N$ has linear volume growth. Since the Ricci curvature of $M$ is positive at some point, the same holds for $\tilde{M}$. This forces $k = 0$, and thus $\tilde{M} = N$ has linear volume growth. 
	
	Applying Lemma~\ref{|Gamma(r)|} to $\Gamma := \pi_1(M)$, we obtain for all sufficiently large $r > 0$:
	\[
	|\Gamma(r)| \le C(n) \frac{\vol(B_r(\tilde{p}))}{\vol(B_r(p))} \le C(n) \frac{Cr}{cr} = C(n)Cc^{-1}.
	\]
	This uniform bound implies that $\Gamma$ is finite.
\end{proof}

\begin{proof}[\textbf{Proof of Corollary \ref{Bounded}}]
	
	By Theorem~\ref{Fibration}, the universal cover $(\tilde{M}, \tilde{p}) \to (M, p)$ admits an isometric splitting $(\tilde{M}, \tilde{p}) \cong (\mathbb{R}^k \times N,(0^k,q))$, with the properties as stated in the theorem. Put $\Gamma=\pi_1(M,p)$.
	
	We first show that there exists a radius $r_0 > 0$ such that $\R^k \times \{q\} \subset \overline{B_{r_0}(\Gamma(\tilde{p}))}$. For any $(x,q)\in\R^k\times \{q\}$, since $\R^k/\Gamma$ is compact, there exists $\gamma\in\Gamma$ such that $d(\gamma(0^k),x)\le \diam (\R^k/\Gamma)$. Then 
	\begin{align*}
		d(\gamma(\tilde p),(x,q))^2=d((\gamma(0^k),\gamma(q)),(x,q))^2&=d(\gamma(0^k),x)^2+d(\gamma(q),q)^2\\&\le \diam(\R^k/\Gamma)^2+\sup_{\gamma'\in\Gamma}d(\gamma'(q),q)^2.
	\end{align*}
	Since the closure of the induced $\Gamma$-action on $N$ is compact (Remark \ref{rem1}), the last term in the above inequality is finite; we denote this value by $r_0^2$. Thus, the first claim follows.
	
	For an arbitrary $\omega\in\Gamma$, let $\sigma_\omega:[0,1]\to\tilde M$ be a shortest geodesic from $\tilde p$ to $\omega(\tilde p)$. Putting $V=\R^k\times\{q\}$, then
	\begin{equation}\label{0:27}
		d(\sigma_\omega(t),V)=td(\omega(q),q)\le\sup_{\gamma'\in\Gamma}d(\gamma'(q),q)< r_0.
	\end{equation}

	Thus, by combining (\ref{0:27}) with the first claim, we have shown that there exists $r_0>0$, such that for every $\omega \in \Gamma$, the shortest geodesic loop representing $\omega$ is contained in $B_{2r_0}(p)$.

\end{proof}

\begin{proof}[\textbf{Proof of Corollary \ref{dim=4}}]
	Since $M$ is not compact and $\pi_1(M)$ is infinite, the first Betti number $b_1(M)\in\{1,2,3\}$.
	
	If $b_1(M)\ge 3$, by \cite{Ye24a}, $M$ is flat.
	
	If $b_1(M) = 2$, then by Remark~\ref{rem1} (2), the universal cover $\tilde{M}$ admits an isometric decomposition $\mathbb{R}^2 \times \mathbb{R}^{k-2} \times N$. Consequently, $M$ has nonnegative sectional curvature. By the Cheeger--Gromoll soul theorem, $M$ is homotopy equivalent to its soul $S$.

	If $b_1(M) = 1$, then by Remark~\ref{rem1} (2), $\tilde{M}$ admits an isometric decomposition $\mathbb{R} \times \mathbb{R}^{k-1} \times N$, where $k$ and $N$ are as described in Theorem \ref{Fibration}. The cases $k = 2$ and $k = 3$ both imply that $M$ has nonnegative sectional curvature. We therefore focus on the case $k = 1$, in which $N$ is a 3-dimensional simply connected manifold with nonnegative Ricci curvature and linear volume growth. By \cite{Liu13}, $N$ is either diffeomorphic to $\mathbb{R}^3$ or isometric to the Riemannian product $S^2 \times \mathbb{R}$, where $S^2$ is a topological sphere with nonnegative sectional curvature. The latter case again implies nonnegative sectional curvature for $M$, so we assume $N$ is diffeomorphic to $\mathbb{R}^3$. By Remark~\ref{rem1}, the commutator subgroup $[\pi_1(M), \pi_1(M)]$ is a subgroup of $\operatorname{Isom}(N)$. Theorem~\ref{Isom} implies that $[\pi_1(M), \pi_1(M)]$ is finite. Combining this with the contractibility of $N$, we conclude that the commutator subgroup is trivial. Therefore, $\pi_1(M)$ is abelian. Moreover, since $\tilde{M}$ is contractible, $\pi_1(M)$ must be free abelian. We conclude that $\pi_1(M) \cong \mathbb{Z}$, and there exists a fiber bundle map $M \to S^1$ with fiber diffeomorphic to $\R^3$, showing that $M$ is homotopy equivalent to $S^1$.

\end{proof}

\section{Splitting and the Infimum of Volume Growth Order $<3$}

\subsection{Proof of Proposition \ref{EffectiveVer.}(2)}
By Lemma~\ref{CriticalScaling}, for every $(X, x, G) \in \Omega(\hat{M}, \Gamma)$, $(X, x)$ splits isometrically as $(\mathbb{R}^{n-2} \times Y, (0^{n-2}, y^*))$, where $Y$ contains no line, and $G(x) = \mathbb{R}^{n-2} \times \{y^*\}$. So by Theorem \ref{escaprate}, $\Gamma$ contains a $\Z^{n-2}$-subgroup with finite index.

Let $H^{1+\epsilon}_0(\hat M)$ be the space of harmonic functions $f:(\hat M,\hat p)\to(\mathbb{R},0)$ satisfying that there exists some $C>0$ such that $|f(x)|\le Cd(x,\hat p)^{1+\epsilon}+C$ for any $x\in\hat M$. Define $W:=\bigcap_{\epsilon>0}H^{1+\epsilon}_0(\hat M)$. By the stability property of equivariant asymptotic cones and \cite[Theorem 1.10]{HH24}, $\dim(W)=n-2$, and there exists a basis $\{u^1,u^2,\ldots,u^{n-2}\}\subset W$ satisfying the following conditions:
\begin{itemize}
	\item For any sufficiently large $R$, there exists a lower diagonal matrix $T_R$ with positive diagonal entries such that $T_R\circ (u^1,\ldots, u^{n-2}): B_R(\hat p)\to \mathbb{R}^{n-2}$ is a $(\Psi(\frac{1}{R}), n-2)$-splitting map, where $\Psi(\frac{1}{R})$ is a function tending to $0$ as $R\to\infty$.
	\item $\|T_R\|\le R^{\Psi(\frac1R)}$.
\end{itemize}

For simplicity, let $\mathbf{u}:=(u^1,u^2,\ldots,u^{n-2})$. Then for every $\gamma\in\Gamma$, there exist a matrix $A\in \mathrm{GL}(n-2,\mathbb{R})$ and a vector $\mathbf{b}\in\mathbb{R}^{n-2}$ such that $\mathbf{u}\circ \gamma = A\mathbf{u} + \mathbf{b}$. In particular, this implies that
\begin{itemize}
	\item $\Gamma$ permutes the level sets of $\mathbf{u}$.
\end{itemize}

For every sequence $r_i \to \infty$, define $(N_i, p_i) := (r_i^{-1} \hat{M}, \hat{p})$ and $\mathbf{u}_i := \frac{T_{6r_i} \mathbf{u}}{r_i}|_{B_{6r_i}(\hat{p})}$. After passing to a subsequence, there exist a group $G$ and a pointed space $(Y, y^*)$ such that
\begin{equation}\label{Conv21:55}
	(N_i, p_i, \Gamma) \stackrel{GH}{\longrightarrow} (X, x^*, G), \quad \mathbf{u}_i \to \mathbf{u}_\infty \text{ in } W^{1,2},
\end{equation}
where $(X, x^*) \cong (\mathbb{R}^{n-2} \times Y, (0^{n-2}, y^*))$ and $\mathbf{u}_\infty$ is the restriction to $B_6(x^*)$ of the standard coordinate functions on the $\mathbb{R}^{n-2}$-factor.

\begin{lem}\label{AchiveMini}
	For every level set $\Sigma$ of $\mathbf{u}$, there exists $\gamma\in\Gamma$ such that 
	\begin{equation}\label{FullfilDistance}
		d(\gamma(\hat p),\Sigma)=\inf\{d(\gamma'(\hat p),\Sigma)|\gamma'\in\Gamma\}.
	\end{equation}
	
\end{lem}
\begin{proof}
	Argue by contradiction. Let $\Sigma = \mathbf{u}^{-1}(\mathbf{x})$ for some $\mathbf{x} \in \mathbb{R}^{n-2}$ satisfy that, for every $\gamma \in \Gamma$, 
	\[
	d(\gamma(\hat{p}), \Sigma) > \inf \{ d(\gamma'(\hat{p}), \Sigma) \mid \gamma' \in \Gamma \}.
	\]
	Choose a sequence of $\gamma_i \in \Gamma$ such that $d(\gamma_i(\hat{p}), \Sigma)$ converges to the infimum. Since $\Gamma$ is discrete, $r_i := d(\gamma_i(\hat{p}), \hat{p}) \to \infty$. Now we consider the convergence (\ref{Conv21:55}). Let $q_i \in \Sigma$ be a projection point of $\gamma_i(\hat{p})$. We can assume $q_i \to q_\infty$ and $\gamma_i \to g \in G$ with respect to (\ref{Conv21:55}) with $q_\infty = g(x^*)$ and $d(g(x^*), x^*) = 1$. So $\mathbf{u}_i(q_i) \to \mathbf{u}_\infty(q_\infty) \in \partial B_1(0^{n-2}) \times \{y^*\}$. However, there exist $\Psi_i\to0$ such that, 
	\[
	\|\mathbf{u}_i(q_i)\| = \left\|\frac{T_{6r_i} \mathbf{u}(q_i)}{r_i}\right\| \le \frac{6^{\Psi_i} \|\mathbf{x}\|}{r_i^{1-\Psi_i}} \to 0, 
	\]
	which is a contradiction.
	
\end{proof}

\begin{lem}\label{LVisclose}

	There exists $D > 0$ such that for every level set $\Sigma$ of $\mathbf{u}$, there is some $\gamma \in \Gamma$ satisfying $d(\gamma(\hat{p}), \Sigma) \le D$.
	
\end{lem}
\begin{proof}

	We argue by contradiction. Assume there exists a sequence of level sets $\Sigma_i$ of $\mathbf{u}$ for which the minimal distances 
	$r_i := \min\{d(\gamma(\hat{p}), \Sigma_i) \mid \gamma \in \Gamma\}$ tend to infinity, 
	where the minimum is well-defined by Lemma~\ref{AchiveMini}. Since $\Gamma$ permutes the level sets, without loss of generality we may assume that there exists $q_i \in \Sigma_i$ such that $d(\hat{p}, q_i) = d(\hat{p}, \Sigma_i) = r_i. $

	We consider the convergence in (\ref{Conv21:55}) for the sequence $r_i$ and assume that $q_i \to q_\infty$. Let $q_\infty = (\mathbf{v}, y) \in \mathbb{R}^{n-2} \times Y$.

	\textbf{Claim:} $\mathbf{v} = 0$. If not, let $x_\infty := (\mathbf{v}, y^*)$. Since $G(x^*) = \mathbb{R}^{n-2} \times \{y^*\}$, there exists $\gamma_\infty \in G$ such that $\gamma_\infty(x^*) = x_\infty$. Choose $\gamma_i \in \Gamma$ with $\gamma_i \to \gamma_\infty$ with respect to the convergence in (\ref{Conv21:55}). Then $d_{N_i}(\gamma_i(p_i), q_i) \to d(x_\infty, q_\infty) = d(y^*, y) < \sqrt{\|\mathbf{v}\|^2 + d(y^*, y)^2} = d(q_\infty, x^*) = 1$.
	
	Thus, there exists $\eta > 0$ such that for all sufficiently large $i$, $d(\hat{p}, \gamma_i^{-1}(q_i)) \le r_i(1 - \eta) < r_i$, which contradicts the definition of $r_i$.

	We will need the following sublemma, whose proof we postpone to the appendix.
	
	\begin{slem}\label{LvGHLv}
		For each $n, v, \epsilon > 0$, there exists $\delta = \delta(n, v, \epsilon) > 0$ satisfying the following. 
		Suppose $M$ is an $n$-manifold with $\mathrm{Ric}_M \ge -\delta$ and $\overline{B_{12}(p)}$ is compact. Let $\mathbf{u}: (B_{6}(p), p) \to (\mathbb{R}^k, 0^k)$ be a $(\delta, k)$-splitting map. If $\vol(B_1(p)) > v$, then $B_1(0^k)\subset \mathbf{u}(B_{1+\epsilon}(p))$.
	\end{slem}

	Continue the proof of the lemma. By the Claim, $\mathbf{x}_i := \mathbf{u}_i(q_i) \to \mathbf{u}_\infty(q_\infty) = 0^{n-2}$. Applying the above sublemma, we conclude that there exists $q_i' \in B_{0.5}(p_i)$ such that $\mathbf{u}_i(q_i') = \mathbf{x}_i = \mathbf{u}_i(q_i)$. Hence $q_i' \in \Sigma_i$ and $d(\hat{p}, \Sigma_i) \le d(\hat{p},q_i') < 0.5r_i < d( \hat{p},q_i)$, which contradicts the choice of $q_i$.
	
\end{proof}

For each level set $\Sigma$ of $\mathbf{u}$, let $\gamma \in \Gamma$ satisfy equation (\ref{FullfilDistance}), and define $\phi(\Sigma) \in \Sigma$ and $\psi(\Sigma) \in \Sigma$ as projections of $\gamma(\hat{p})$ and $\hat{p}$, respectively.

\begin{lem}\label{Lem0:52}
	The ratio $\frac{d(\phi(\Sigma), \psi(\Sigma))}{d(\hat{p}, \Sigma)}$ tends to $0$ as $d(\hat{p}, \Sigma)$ tends to $\infty$.
\end{lem}
\begin{proof}
	Argue by contradiction. Assume there exists a sequence of level sets $\Sigma_i$ of $\mathbf{u}$ such that
	\[
	\liminf_{i\to\infty}\frac{d(\phi(\Sigma_i), \psi(\Sigma_i))}{d(\hat{p}, \Sigma_i)}>0, \quad \lim_{i\to\infty} d(\hat{p}, \Sigma_i) = \infty.
	\]
	
	Let $r_i := d(\hat{p}, \phi(\Sigma_i))$, $q_i := \phi(\Sigma_i)$, and $q_i' := \psi(\Sigma_i)$. Choose $\gamma_i \in \Gamma$ such that $\gamma_i$ satisfies equation (\ref{FullfilDistance}) for $\Sigma_i$, and $q_i$ is a projection point of $\gamma_i(\hat{p})$ to $\Sigma_i$.
	
	Now we again consider convergence (\ref{Conv21:55}) for $r_i$ defined as above, and assume that $q_i \to q_\infty$ and $q_i' \to q_\infty'$.
	
	 By Lemma~\ref{LVisclose}, we have
	\[
	d(\gamma_i(\hat{p}), q_i) \le D, \quad r_i - D \le d(\hat{p}, \gamma_i(\hat{p})) \le r_i + D.
	\]
	Thus we can assume $\gamma_i \to g \in G$ with respect to the convergence in (\ref{Conv21:55}). Consequently, $q_i \to q_\infty = g(x^*) \in \partial B_1(0^{n-2}) \times \{y^*\}$ and
	\[
	\mathbf{u}_i(q_i') = \mathbf{u}_i(q_i) \to \mathbf{u}_\infty(q_\infty') = \mathbf{u}_\infty(q_\infty) =: \mathbf{x} \in \partial B_1(0^{n-2}).
	\]
	
	Let $q_\infty' = (\mathbf{x}, y)$ with $\|\mathbf{x}\| = 1$. Since $d_{N_i}(p_i, q_i) \ge d_{N_i}(p_i, q_i')$,
	\[
	1 = d(x^*, q_\infty) \ge d(x^*, q_\infty') = \sqrt{\|\mathbf{x}\|^2 + d(y, y^*)^2} \ge 1,
	\]
	which forces $y = y^*$. In conclusion, $q_\infty = q_\infty'$ and $r_i^{-1}d(q_i, q_i') \to 0$, which implies
	\[
	 \frac{d(q_i, q_i')}{d(\hat{p}, \Sigma_i)}\le \frac{d(q_i, q_i')}{d(\hat p,q_i)-d(q_i,q_i')} \to 0.
	\]
	This contradicts the initial assumption.
\end{proof}

The key ingredient is the slicing Theorem due to Cheeger-Naber.

\begin{thm}\cite[Theorem 1.23]{CN15}\label{slicing-thm}
	For each $\epsilon > 0$, there exists $\delta(n,\epsilon) > 0$ such that if $M^n$ satisfies $\Ric_{M^n} \geq-(n-1)\delta$ and if $u : B_{2}(p)\rightarrow\R^{n-2}$ is a harmonic $\delta$-splitting map, then there exists a subset $G_{\epsilon}\subseteq B_{1}(0^{n-2})$ that satisfies the following:
	\begin{enumerate}
		\item $\vol(G_{\epsilon}) > \vol(B_{1}(0^{n-2}))-\epsilon$;
		\item if $\mathbf{x}\in G_{\epsilon}$, then $u^{-1}(\mathbf{x})$ is nonempty;
		\item for each $x \in u^{-1}(G_{\epsilon})$ and $r\leq 1$, there exists a lower triangular matrix $A\in \mathrm{GL}(n-2)$ with positive diagonal entries such that $A\circ u : B_{r}(x)\rightarrow\R^{n-2}$ is an $\epsilon$-splitting map.
	\end{enumerate}

\end{thm}

Now we are ready to prove part (2) of Proposition~\ref{EffectiveVer.}. Fix an arbitrary large $r > 0$. For any sufficiently large $R > r$, we apply the slicing theorem to the rescaled map 
\[
\mathbf{u}_R := R^{-1}T_{6R} \circ \mathbf{u}|_{B_6(\hat{p}, R^{-1}\hat{M})}.
\]
We can find a vector $\mathbf{x} \in B_{0.2}(0^{n-2}) \setminus \overline{B_{0.1}(0^{n-2})}$ such that:
\begin{itemize}
	\item The preimage $\mathbf{u}_R^{-1}(\mathbf{x}) \cap B_{0.4}(\hat{p}, R^{-1}\hat{M})$ is nonempty;
	\item For every $x \in \mathbf{u}_R^{-1}(\mathbf{x}) \cap B_{1}(\hat{p}, R^{-1}\hat{M})$ and every $s \leq 1$, the geodesic ball $B_s(x, R^{-1}\hat{M})$ is $\Psi(1/R) s$-close to a geodesic ball of radius $s$ in $\mathbb{R}^{n-2} \times Y$ for some $Y$,
\end{itemize}
where the latter conclusion follows from the quantitative splitting theorem \cite{CC96} and property (3) of Theorem~\ref{slicing-thm}.

Let $\Sigma_R := \mathbf{u}^{-1}(R T_{6R}^{-1} \mathbf{x})$. Then $\psi(\Sigma_R) \in B_{0.4}(\hat{p}, R^{-1}\hat{M}) \setminus \overline{B_{0.09}(\hat{p}, R^{-1}\hat{M})}$. Applying Lemma~\ref{Lem0:52}, we conclude that $\phi(\Sigma_R) \in B_1(\hat{p}, R^{-1}\hat{M})$. Hence $\phi(\Sigma_R)$ satisfies that for any $s \leq R$, the geodesic ball $B_s(\phi(\Sigma_R))$ is $\Psi(1/R) s$-close to a geodesic ball of radius $s$ in $\mathbb{R}^{n-2} \times Y$ for some $Y$. By Lemma~\ref{LVisclose}, there exists $\gamma \in \Gamma$ such that $d(\gamma(\phi(\Sigma_R)), \hat{p}) \le D$. Noting that $B_r(\hat{p}) \subset B_{2r}(\gamma(\phi(\Sigma_R)))$, we conclude that $B_r(\hat{p})$ is $\Psi(1/R) r$-close to a geodesic ball of radius $r$ in $\mathbb{R}^{n-2} \times Y$ for some $Y$. Letting $R \to \infty$ and then $r \to \infty$, we conclude that $\hat{M}$ splits off an $\mathbb{R}^{n-2}$-factor isometrically. This completes the proof.

\subsection{Proof of Theorem \ref{Quadra}}

By \cite[Corollary 4]{KW11}, $\pi_1(M)$ contains a nilpotent subgroup of index $\le C(n)$. Hence up to a finite cover, \textbf{we assume $\Gamma:=\pi_1(M)$ is nilpotent}. In particular, for any $(Y,y,G)\in\Omega(\tilde M,\Gamma)$, $G$ is nilpotent.

By Cheeger--Colding \cite{CC96}, for any $(Y, y, G) \in \Omega(\tilde{M}, \Gamma)$, the space $(Y, y)$ is isometric to some $(\mathbb{R}^k \times C(Z), (0^k, z^*))$, where $C(Z)$ is the metric cone over $Z$ with $\operatorname{diam}(Z) < \pi$. In particular, the orbit $G(y)$ is a closed submanifold of $\mathbb{R}^k \times \{z^*\}$. If $k\ge n-1$, by Cheeger-Colding's codimension-$2$ regularity theorem \cite{CC97} and Colding's volume rigidity theorem \cite{Col97}, then $M$ is flat and the required conclusion follows immediately. \textbf{Hence, in what follows, we assume that $k\le n-2$ for every $Y$.}

Let $a:=\liminf_{r\to\infty}\frac{\vol(B_r(p))}{r^{3-\delta}}<\infty$ and $v:=\lim_{r\to\infty}\frac{\vol(B_r(\tilde p))}{r^n}>0$, where $\tilde p\in\tilde M$ is a preimage of $p$.
By Lemma~\ref{|Gamma(r)|},
\[
c_nv\frac{r^n}{\vol(B_r(p))}\le c_n \frac{\vol(B_r(\tilde{p}))}{\vol(B_r(p))} \le |\Gamma(r)|.
\]
Hence,
\begin{equation}\label{LiminfNondegenerate}
	\frac{c_nv}{a}\le  \limsup_{r\to\infty}\frac{|\Gamma(r)|}{r^{n-3+\delta}}.
\end{equation}
By \cite[Proposition 2.2]{Ye24}, there exists a triple $(X, x, H) \in \Omega(\tilde{M}, \Gamma)$ such that the orbit $H(y)$ has lower box dimension at least $n-3+\delta$. This implies that the maximal Euclidean factor of $X$ has dimension at least $n-2$.
Then $(X, x)$ is isometric to $(\mathbb{R}^{n-2} \times C(S_r^1), (0^{n-2}, z^*))$, where $S_r^1$ denotes a circle of radius $r \in (0,1)$.
Since $H(x)$ is a closed submanifold of the $\mathbb{R}^{n-2}$-slice through $x$, it must coincide with the entire $\mathbb{R}^{n-2}$-slice.

According to Lemma~\ref{CriticalScaling} and the above discussion, we conclude that, for every $(Y, y,G) \in \Omega(\tilde{M},\Gamma)$, $(Y,y)$ is isometric to $(\R^{n-2} \times C(S_r^1),(0^{n-2},z^*))$ and $G(y)=\R^{n-2}\times\{z^*\}$. Moreover, by Colding's volume convergence theorem \cite{Col97}, the radius $r$ is determined solely by the volume ratio $v$ of $\tilde{M}$ and is independent of the choice of $(Y, y)$.

Now we can apply Proposition~\ref{EffectiveVer.}(2) to obtain the splitting of $\tilde{M}$. In particular, $M$ has nonnegative sectional curvature. We can also apply Theorem~\ref{escaprate} to conclude that $\pi_1(M)$ contains a $\mathbb{Z}^{n-2}$-subgroup. Hence the soul $S$ of $M$ has dimension $n-1$ or $n-2$; in either case, the universal cover of $S$ is an Euclidean space.

\subsection{Proof of Corollary \ref{b_1>n-3}}

	Define 
	\[
	\Gamma = \left( \pi_1(M) / [\pi_1(M),\pi_1(M)] \right) / T, \quad 
	(\hat{M},\hat p) := \left( (\tilde{M},\tilde p) / [\pi_1(M),\pi_1(M)] \right) / T,
	\]
	where $T$ is the torsion subgroup and $\tilde p\in\tilde M$ is a preimage of $p$. Then $(\hat{M},\hat{p}) \to (M,p)$ is a normal cover with deck transformation $\Gamma$, where $\Gamma$ is torsion-free abelian. A priori, we do not know whether $\Gamma$ is finitely generated, but it must contain a $\mathbb{Z}^{b_1}$-subgroup. By Lemma~\ref{18:09}, for large $r>0$,
	\begin{equation*}\label{upperboundorder2}
		|\Gamma(r)| \ge c r^{b_1}, \quad \vol(B_r(p)) \le C r^{n-b_1}\le Cr^{2}.
	\end{equation*}
	If $\tilde{M}$ has Euclidean volume growth, then by Theorem~\ref{Quadra}, the conclusion follows directly.
	
Suppose that $\tilde{M}$ does not have Euclidean volume growth. Then for any $(Y, y, G) \in \Omega(\hat{M}, \Gamma)$, the rectifiable dimension of $Y$ is at most $n-1$. By a standard diagonal argument and \cite[Corollary 3.2]{PY24}, the rectifiable dimension of $Y/G$ is $1$. Consequently, every $(X, x) \in \Omega(M)$ has rectifiable dimension $1$. We now require the following well-known fact, whose proof is deferred to the appendix.

\begin{slem}\label{1-dim}
	If every asymptotic cone of an open manifold $N$ with nonnegative Ricci curvature has rectifiable dimension $1$, then $\Omega(N)$ consists of a single element, isometric to $(\mathbb{R},0)$ or $(\mathbb{R}_{\ge 0},0)$.
\end{slem}

By the above sublemma, the finite generation of $\pi_1(M)$ follows from Sormani's pole group theorem \cite{Sor99}.
 
\section*{Appendix: Proofs of Some Lemmas}

\subsection{Proof of Sublemma \ref{sublem}}

Since $\Gamma \cong \mathbb{Z}^k$ is a closed subgroup of $\operatorname{Isom}(\mathbb{R}^n)$ and contains no nontrivial compact subgroups, it follows that its action on $\mathbb{R}^n$ is free. Consequently, the quotient $\mathbb{R}^n / \Gamma$ is a flat open manifold whose soul is a $k$-torus $T^k$. It follows that $\mathbb{R}^n$ admits a metric decomposition as $\mathbb{R}^k \times \mathbb{R}^{n-k}$ such that, after an appropriate choice of origin, the subspace $\mathbb{R}^k \times \{0^{n-k}\}$ is the preimage of $T^k$ under the covering projection $\mathbb{R}^n \to \mathbb{R}^n / \Gamma$. This implies that $\mathbb{R}^k \times \{0\}$ is $\Gamma$-invariant, and the $\Gamma$-action splits isometrically over the product $\mathbb{R}^k \times \mathbb{R}^{n-k}$: it acts freely by translations on the $\mathbb{R}^k$-factor and fixes the origin of the $\mathbb{R}^{n-k}$-factor.

We now consider the action of the full group $G$ on the product decomposition $\mathbb{R}^k\times\mathbb{R}^{n-k}$ obtained above. For any $\gamma\in\Gamma\setminus\{e\}$, set $v:=\gamma(0^k)$. Define the line $\sigma(t):=(tv,0^{n-k})$ in $\mathbb{R}^k\times\mathbb{R}^{n-k}$. Note that $\Gamma$ acts freely by translations on the invariant subspace $\mathbb{R}^k \times \{0^{n-k}\}$, so for every integer $l$, we have $\gamma^l(0^k,0^{n-k})=\sigma(l)$. 

Let $\omega\in G$ and assume $\omega(0^k,0^{n-k})=:(a,b)\in\mathbb{R}^k\times\mathbb{R}^{n-k}$. Then there exists $(\alpha,\beta)\in\mathbb{R}^k\times\mathbb{R}^{n-k}$ such that $\omega(\sigma(t))=(a+t\alpha,b+t\beta)$. Since $\Gamma$ is normal in $G$, there exists $\gamma_1\in\Gamma$ such that $\gamma_1=\omega\gamma\omega^{-1}$. Thus for any integer $l$,
\[
(a+l\alpha,b+l\beta)=\omega(\sigma(l))=\omega\gamma^l(0^k,0^{n-k})=\gamma_1^l\omega(0^k,0^{n-k})=(\gamma_1^l(a),\gamma_1^l(b)).
\]
Recall that $\Gamma$ acts on $\mathbb{R}^{n-k}$ fixing the origin, so $\gamma_1^l(b)\in B_{\|b\|}(0^{n-k})$ for every $l$. This forces $\beta=0$, i.e., $\omega(tv,0^{n-k})=(a+t\alpha,b)$.

By choosing elements $\gamma\in\Gamma\setminus\{e\}$ such that the corresponding vectors $v$ span $k$ linearly independent directions, we see that $\omega$ maps $k$ linearly independent lines in $\mathbb{R}^k\times\{0^{n-k}\}$ to $\mathbb{R}^k\times\{b\}$. This shows that $\omega$ acts separately on $\mathbb{R}^k\times\mathbb{R}^{n-k}$.

\subsection{Proof of Sublemma \ref{LvGHLv}}

	We argue by contradiction. Assume that for some $\epsilon, n, v > 0$, there exist a sequence $(M_i, p_i)$ and maps $\mathbf{u}_i: (B_{6}(p_i), p_i) \to (\mathbb{R}^k, 0^k)$ satisfying the conditions of the lemma with $\delta_i \to 0$, but for each $i$, there exists $\mathbf{x}_i \in B_1(0^k)$ such that $\mathbf{u}_i^{-1}(\mathbf{x}_i) \cap B_{1+\epsilon}(p_i) = \emptyset$.
	
	By the quantitative splitting theorem \cite{CC96}, there exist a length space $(Z_i, z_i)$ and a map $\phi_i: (B_3(p_i), p_i) \to (Z_i, z_i)$ such that the combined map $$(\mathbf{u}_i, \phi_i): B_3(p_i) \to B_{3+\epsilon_i}((0^k, z_i))\subset \R^k\times Z_i$$ is an $\epsilon_i$-GHA (Gromov-Hausdorff approximation), for some $\epsilon_i\to0$. In particular, there exists $q_i\in B_{1}(p_i)$ satisfying $\|\mathbf{u}_i(q_i)-\mathbf{x}_i\|\le \epsilon_i$. For each $i$, define 
	$$\mathbf{u}_i':=\epsilon_i^{-1}(\mathbf{u}_i-\mathbf{u}_i(q_i)).$$
	Then we can choose a new sequence $\epsilon_i \to 0$ sufficiently slowly so that the following properties hold:
	\begin{itemize}
		\item The scaled map
		\[
		(\mathbf{u}_i',\phi_i):B_{\epsilon_i^{-1}}(q_i,\epsilon_i^{-1}M_i)\to B_{\epsilon_i^{-1}+\epsilon_i}((0^k,z_i))\subset \R^k\times\epsilon_i^{-1}Z_i,
		\]
		which sends $q_i$ to $(0^k,z_i)$, is an $\epsilon_i$-GHA.
		
		\item The map $\mathbf{u}_i':B_3(q_i,\epsilon_i^{-1}M_i)\to\R^k$ is an $(\epsilon_i,k)$-splitting map.
		
		\item With $\mathbf{x}_i':=\epsilon_i^{-1}(\mathbf{x}_i-\mathbf{u}_i(q_i))$, we have $\|\mathbf{x}_i'\|\le\epsilon_i$.
	\end{itemize}
	
A contradiction arises once we establish that:
\begin{itemize}
	\item [\textbf{Claim.}] For every sufficiently large $i$, there exists $y_i' \in B_{1}(q_i, \epsilon_i^{-1}M_i)$ such that $\mathbf{u}_i'(y_i') = \mathbf{x}_i'$.
	
\end{itemize}

After passing to a subsequence, we assume that
	\begin{equation}\label{Conv21:03}
		(\epsilon_i^{-1}M_i, q_i) \stackrel{GH}{\longrightarrow} (Y, y^*), \quad \mathbf{u}_i' \myarrow{W^{1,2}} P,
	\end{equation}
	where $(Y,y^*)$ is isometric to $(\R^k\times Z,(0^k,z^*))$ for some non-collapsed $\mathrm{RCD}(0,n-k)$-space and $P:(Y,y^*)\to(\R^k,0^k)$ is the standard projection. Note that $\lim_{i\to\infty}\mathbf{x}_i'=0^k$. Since regular points are dense in $Z$, we can choose a regular point $y \in \{0^k\}\times B_{0.1}(z^*)$ and choose $y_i \in B_{0.2}(q_i,\epsilon_i^{-1}M_i)$ such that $y_i \to y$ with respect to the convergence in (\ref{Conv21:03}). Hence $\mathbf{u}_i'(y_i)\to P(y)=0^k$ which implies 
	\begin{equation}\label{21:07}
		s_i:=\|\mathbf{u}_i'(y_i)-\mathbf{x}_i'\|\to0.
	\end{equation}
	
	Since $y$ is a regular point, for any sufficiently small $\delta > 0$, we can choose $r = r_\delta \in (0, 0.01)$ such that
	\[
	d_{GH}(B_{9r}(y), B_{9r}(0^n)) \le \delta r.
	\]
	In turn, for each sufficiently large $i$, we can find a map $$\mathbf{v}_i: (B_r(y_i,\epsilon_i^{-1}M_i), y_i) \to (\mathbb{R}^{n-k}, 0^{n-k})$$ such that the combined map
	\[
	f_i := (\mathbf{u}_i' - \mathbf{u}_i'(y_i), \mathbf{v}_i): (B_r(y_i,\epsilon_i^{-1}M_i), y_i) \to (\mathbb{R}^n, 0^n)
	\]
	is a $(\Psi(\delta), n)$-splitting map, where $\Psi(\delta) \to 0$ as $\delta \to 0$. By the canonical Reifenberg theorem \cite{CJN21}, $f_i$ is a diffeomorphism onto its image and satisfies $B_{0.9r}(0^n) \subset f_i(B_r(y_i,\epsilon_i^{-1}M_i))$ provided $\delta$ is chosen sufficiently small only depending on $n$. Note that by (\ref{21:07}), we have
	\[
	(\mathbf{x}'_i - \mathbf{u}'_i(y_i), 0^{n-k}) \in \overline{B_{s_i}(0^n)} \subset B_{0.9r}(0^n).
	\]
	Hence, there exists a point $y_i' \in B_r(y_i,,\epsilon_i^{-1}M_i)$ such that $f_i(y_i') = (\mathbf{x}'_i - \mathbf{u}'_i(y_i), 0^{n-k})$, which implies $\mathbf{u}'_i(y_i') = \mathbf{x}'_i$. Note that $y_i' \in B_r(y_i,\epsilon_i^{-1}M_i) \subset B_{0.2+ r}(q_i,\epsilon_i^{-1}M_i) \subset B_{1}(q_i,\epsilon_i^{-1}M_i)$. This proves the claim.
	
\subsection{Proof of Sublemma \ref{1-dim}}
	
By \cite{Hon13} (see also \cite{ChLi16}), for any $(X, x) \in \Omega(N)$, the space $(X, x)$ is isometric to either $(\mathbb{R}, 0)$ or $(\mathbb{R}_{\ge 0}, a)$ for some $a \ge 0$. However, the case $(\mathbb{R}_{\ge 0}, a)$ with $a > 0$ is impossible. Admitting this, the conclusion follows from the connectedness of $\Omega(N)$.

If not, fix a point $p \in N$. Then we can choose sequences $o_i, q_i \in N$ and $r_i \to \infty$ such that
\[
(r_i^{-1} N, o_i, p, q_i) \stackrel{\mathrm{GH}}{\longrightarrow} (\mathbb{R}_{\ge 0}, 0, a, 2a).
\]

Let $\gamma_i$ be a shortest geodesic from $o_i$ to $q_i$. It is clear that
\[
\lim_{i \to \infty} \frac{d_N(p, \gamma_i)}{d_N(o_i, q_i)} = 0.
\]

We now consider two cases:

\textbf{Case 1:} $\liminf_{i \to \infty} d_N(p, \gamma_i) < \infty$. This implies that $N$ splits off a line, which leads to a contradiction.

\textbf{Case 2:} $\lim_{i \to \infty} d_N(p, \gamma_i) = \infty$. Let $s_i := d_N(p, \gamma_i)$. After passing to a subsequence, we have
\[
(s_i^{-1} N, p,\gamma_i) \stackrel{\mathrm{GH}}{\longrightarrow} (X, x,l).
\]
We conclude that $X$ contains a line $l$ that does not pass through $x$, which contradicts the assumption.

\vspace*{20pt}

\noindent\textbf{Acknowledgments.} The authors thank Dr. Zhu Ye for asking whether the proof of an earlier version of Theorem~\ref{Quadra} implied Corollary~\ref{b_1>n-3}, which motivated the improvements in the current versions of these results.

\bibliographystyle{alpha}
\bibliography{ref}

\end{document}